\documentclass{amsart}
\usepackage{graphicx, amssymb, eucal, hyperref, color}
\newtheorem{thm}{Theorem}[section]

\newtheorem{lem}[thm]{Lemma}

\newtheorem{prop}[thm]{Proposition}
\newtheorem{prob}[thm]{Problem}
\theoremstyle{definition}
\newtheorem{defn}[thm]{Definition}

\theoremstyle{remark}

\newtheorem{ex}[thm]{Example}

\numberwithin{equation}{section}
\newcommand{\vertiii}[1]{{\left\vert\kern-0.25ex\left\vert\kern-0.25ex\left\vert #1 
    \right\vert\kern-0.25ex\right\vert\kern-0.25ex\right\vert}}

\newcommand{\defined}[1]{\emph{#1}}
\newcommand{\mc}[1]{\mathcal{#1}}
\newcommand{\op}[1]{\operatorname{#1}}
\newcommand{\ignore}[1]{}

\newcommand{\mb}[1]{\mathbf{#1}}

\newcommand{\tp}{\operatorname{tp}}

\renewcommand{\phi}{\varphi}

\begin{document}

\title[Omitting Types and the Baire Category Theorem]{Omitting Types and the Baire Category Theorem}%
\author[C. J. Eagle]{Christopher J. Eagle${}^1$} 
\address[C. J. Eagle]{University of Victoria, Department of Mathematics and Statistics, PO BOX 1700 STN CSC, Victoria, British Columbia, Canada, V8W 2Y2}%
\email{eaglec@uvic.ca}
\urladdr{http://www.math.uvic.ca/~eaglec/}

\author[F. D. Tall]{Franklin D. Tall${}^1$}
\address[F. D. Tall]{University of Toronto, Department of Mathematics, 40 St. George St., Toronto, Ontario, Canada M5S 2E4}
\email{f.tall@utoronto.ca}
\urladdr{http://www.math.toronto.edu/tall/}
\thanks{${}^1$ Research supported by NSERC Grant A-7354}


\date{\today}%
% ----------------------------------------------------------------
\begin{abstract}
The Omitting Types Theorem in model theory and the Baire \\Category Theorem in topology are known to be closely linked.  We examine the precise relation between these two theorems.  Working with a general notion of logic we show that the classical Omitting Types Theorem holds for a logic if a certain associated topological space has all closed subspaces Baire.  We also consider stronger Baire category conditions, and hence stronger Omitting Types Theorems, including a game version.  We use examples of spaces previously studied in set-theoretic topology to produce abstract logics showing that the game Omitting Types statement is consistently not equivalent to the classical one.
\end{abstract}
\maketitle
% ----------------------------------------------------------------
\section*{Introduction}\label{sec:Introduction}
The Baire Category Theorem asserts that, in certain classes of spaces, the intersection of countably many dense open sets is dense.  This theorem is used in many fields of mathematics to assure the existence of objects satisfying countably many requirements simultaneously.  It is well known that (locally) compact spaces and completely metrizable spaces satisfy the Baire Category Theorem.  General topologists have developed a finely graded collection of other topological properties that suffice to prove the Baire Category Theorem. \cite{Aarts1974} is a good survey of many of them. Spaces satisfying the conclusion of the Baire Category Theorem are called \emph{Baire spaces}.  They do not behave well with respect to topological operations -- notoriously, with respect to products. See \cite{Fleissner1978}. 

The Omitting Types Theorem for first-order logic asserts that for any countable theory $T$, a type of $T$ can be omitted in a model of $T$ provided that it is not implied by a single formula (modulo $T$).  One way to prove the Omitting Types Theorem is to translate the relevant definitions into the language of topology, and then apply the Baire Category Theorem (the relevant spaces are compact); for instance, see \cite{Poizat2000}.  This method is also able to prove omitting types results in more general settings, including countable fragments of $L_{\omega_1, \omega}$ \cite{Morley1974}, continuous logic \cite{Caicedo2014}, and countable fragments of $L_{\omega_1, \omega}$ for metric structures \cite{Eagle2014}.

It is natural to ask whether the use of Baire spaces in the proofs of the omitting types theorems mentioned above is in any sense essential.  To address this question we consider an abstract topological setting in which core notions from model theory can be developed.  This point of view was first used by Morley \cite{Morley1974}.  A similar setting was also used by Ben Yaacov \cite{BenYaacov2003} as part of his study of \emph{cats}, and also by Knight \cite{Knight2007} (and in simplified form, \cite{Knight2010}).  In this framework we describe three possible \emph{omitting types properties} an abstract logic might have, and relate those properties to topological properties (and more specifically, properties related to Baire category) of the type-spaces of the logic.  

As mentioned above, the property of being a Baire space is not well-behaved with respect to many topological operations.  In Section \ref{sec:TypeSpace:Space} we construct, from any topological space $X$, an abstract logic that has $X$ as its space of $1$-types.  We then use examples of this kind to show in Section \ref{sec:Distinguishing} that the strongest of the omitting types properties we define is consistently strictly stronger than the others.  We conclude with some remarks on the descriptive set-theoretic complexity required of examples separating our omitting types properties.           

\subsection*{Acknowledgements}
Part of this paper was written during the Focused Research Group ``Topological Methods in Model Theory" at the Banff International Research Station.  We thank BIRS for providing an excellent atmosphere for collaboration and research, and we appreciate the comments we received from Xavier Caicedo, Eduardo Du\'e\~{n}ez, and Jos\'e Iovino during the Focused Research Group.

\section{Type-space functors and abstract logics}\label{sec:TypeSpace}
In order to explore the connection between Baire Category and Omitting Types, we need a suitable framework in which to connect topological spaces to logics.  Specifically, we consider collections of topological spaces satisfying some conditions similar to those enjoyed by the type-spaces of a first order theory.  This idea goes back to Morley \cite{Morley1974}, but our primary reference for definitions is \cite{Knight2007} (though our definitions do vary somewhat from the ones presented there; see below for details).  While it is possible to use this setting to produce abstract logics (in the sense of \cite{Barwise1985}), we will not need to do so for our purposes in this note.

\subsection{Basic definitions}\label{sec:TypeSpace:Definitions}
To simplify notation, whenever $\kappa$ is a cardinal and $A \in [\kappa]^n$, we write $A = \{a_0 < \ldots < a_{n-1}\}$ to mean that $A = \{a_0, \ldots, a_{n-1}\}$ and $a_0 < \ldots < a_{n-1}$.

\begin{defn}
A \defined{type-space functor} is a contravariant functor $S$ from $\omega$ to the category of topological spaces with continuous open maps, satisfying a weak amalgamation property.  Explicitly, $S$ takes each $n \in \omega$ to a topological space $S_n$, and each $f : n \to m$ to a continuous open map $Sf : S_m \to S_n$, satisfying the following conditions.  Here $i_k : k \to k+1$ is the inclusion, and $d_m : m+1 \to m+2$ is $d(j) = j$ for $j < m$ and $d(m) = m+1$.
\begin{enumerate}
\item{
For all $f : n \to m$ and $g : m \to k$, $S(g \circ f) = (Sf) \circ (Sg)$.
}
\item{
If $\iota_n : n \to n$ is the identity function then $S\iota_n : S_n \to S_n$ is the identity function.
}
\item{
For each $m \in \omega$, $p \in S_m$, $q \in (Si_m)^{-1}(\{p\})$, and non-empty open $U \subseteq (Si_m)^{-1}(\{p\}$, let $WAP_S(m, p, q, U)$ be the statement that there is $r \in S_{m+2}$ such that $(Si_{m+1})(r) = q$ and $(Sd_m)(r) \in U$.  We require that $WAP_S(m, q, p, U)$ holds for all such $m, p, q, U$.
}
\end{enumerate}
\end{defn}

Our definition is based on the one in Knight \cite{Knight2007}, with some of the simplifications introduced in \cite{Knight2010}.  We differ from Knight in that we require each map $Sf$ to be open, we only require a weak version of the amalgamation property, and we do not require the spaces $S_n$ to be $0$-dimensional.  

The motivating example of a type-space functor is the collection of type-spaces of a theory $T$ (in first order logic, or more generally in a fragment of $L_{\omega_1, \omega}$), where for $f : n \to m$ the map $Sf : S_m(T) \to S_n(T)$ is defined by
\[(Sf)(p) = \{\phi(x_0, \ldots, x_{n-1}) : \phi(x_{f(0)}, \ldots, x_{f(n-1)})\}.\]
We will consider this example in more detail in Section \ref{sec:Examples} below, after introducing more definitions.  For the moment we note that Morley \cite{Morley1974} showed that if a type-space functor has each $S_n$ a 0-dimensional Polish space, and if a stronger amalgamation condition holds, then $S$ arises from a theory in a countable fragment of $L_{\omega_1, \omega}$ in the way described above, and moreover the theory obtained is essentially unique.  Ben Yaacov \cite{BenYaacov2005} showed that without the 0-dimensionality assumption it is still often possible to give a syntactic presentation of a type-space functor, but the associated logic is that of metric structures (see also \cite{BenYaacov2008a}).

Later we will make use of an analogue of the space of $\omega$-types (that is, types in a fixed $\omega$-sequence of variables) for an arbitrary type-space functor.  As type-space functors are designed to capture the model theory of languages in which formulas have finitely many free variables, we expect that knowing the topology of each $S_n$ should suffice to determine the topology of $S_{\omega}$.  Our next definition ensures that this is the case.

\begin{defn}
Let $S$ be a type-space functor.  Define $S_{\omega}$ to be the inverse limit of the spaces $S_n$, using each $S\iota_{n, m}$ as a bonding map for $n < m$.  Concretely,
\[S_{\omega} = \{(a_n)_{n < \omega} \in \prod_{n < \omega} S_n : \text{ for all $n < m$, $a_n = (S\iota_{n, m})(a_m)$}\},\]
with the subspace topology.

For each map $f : n \to \omega$ we have a map $Sf : S_\omega \to S_n$.  To define this map, let $m$ be large enough so that the image of $f$ is included in $m$.  Then define $f' : n \to m$ to be $f'(i) = f(i)$ for all $i < n$.  Finally, define $Sf : S_\omega \to S_n$ by $(Sf)((a_j)_{j < \omega}) = (Sf')(a_m)$.
\end{defn}

The maps $Sf : S_\omega \to S_n$ are well-defined.  Indeed, suppose that we chose another $m'$ to use in defining $Sf$, and without loss of generality assume that $m < m'$.  Defining $f'' : n\to m'$ as above, we then have $f'' = \iota_{m, m'} \circ f'$, and hence
\[(Sf'')(a_{m'}) = S(\iota_{m, m'} \circ f')(a_{m'}) = (Sf') \circ (S\iota_{m, m'})(a_{m'}) = (Sf')(a_m),\]
where the last equality follows from the definition of $S_\omega$.  One equally readily verifies that the maps are continuous and open.\hfill\qed

\subsection{Model theory}\label{sec:TypeSpace:ModelTheory}
In this section we describe how to view a type-space functor as a general setting in which to study model theory.  The key definition is that of a \emph{model}, which we take from \cite[Definition 2.9]{Knight2007}.

\begin{defn}\label{def:Model}
Let $S$ be a type-space functor, and let $\kappa$ be a cardinal.  A \defined{model of size $\kappa$} for $S$ is a function $M$, whose domain is $[\kappa]^{<\omega}$, satisfying the list of properties below for all $A = \{a_0 < \ldots < a_{n-1}\} \in [\kappa]^n$.
\begin{enumerate}
\item{
$M(A) \in S_n$.
}
\item{
If $B = \{b_0 < \ldots < b_{m-1}\} \in [\kappa]^m$, $A \subseteq B$, and $g : n \to m$ is the function satisfying $a_i = b_{g(i)}$ for all $i$, then $M(A) = (Sg)(M(B))$.
}
\item{
If $U \subseteq (S\iota_{n, m})^{-1}(M(A))$ is open, then there is $B = \{b_0 < \ldots < b_{m-1}\} \in [\kappa]^m$ with $A \subseteq B$, and a permutation $g$ of $m$ satisfying $a_i = b_{g(i)}$ for all $i < n$, such that $(Sg)(M(B)) \in U$.
}
\end{enumerate}
\end{defn}

The intuition for this definition is that if $T$ is a first-order theory, and $\mc{M}~\models~T$ is enumerated as $\{m_{\alpha} : \alpha < \kappa\}$, then defining 
\[M(\{i_0 < \ldots < i_{n-1}\}) = \tp^M(m_{i_0}, \ldots, m_{i_{n-1}})\]
gives a model in the above sense.  In fact, when $S = S(T)$ is the type-space functor of a first-order theory, then every model for $S$ arises in this way \cite[Proposition 2.10]{Knight2007} (the proof there is  for countable models, but the argument for general $\kappa$ is the same).

Once we have models, we want to have a satisfaction relation, which we adapt from \cite[Definition 2.11]{Knight2007}.  We do not have a notion of formulas for type-space functors, but we think of elements of $S_n$ as (complete) $n$-types, and so we can define what it means for a finite subset of $\kappa$ to realize an element of $S_n$, in a way that generalizes the case of a tuple in a model realizing a type in first-order logic\footnote{It is not unusual to consider model theory where we have a notion of \emph{type}, but not one of \emph{formula}; for instance, this is the case in Abstract Elementary Classes \cite{Shelah1987}.}.  We could introduce a notion of formula by fixing, for each $n$, a base of closed sets for the space $S_n$, and thinking of basic closed sets as representing formulas while arbitrary closed sets are (partial) types.  We will not pursue this direction further in this note.

\begin{defn}\label{def:Satisfaction}
Let $S$ be a type-space functor, let $M$ be a model for $S$ of size $\kappa$, and let $(a_0, \ldots, a_{n-1})$ be a tuple of length $n$ from $\kappa$.  Let $A=\{a_0,~\ldots,~a_{n-1}\}=\{c_0~<~\ldots~<~c_{k-1}\}$.  Let $g : n \to k$ be the function such that $c_{g(i)} = a_i$ for all $i < n$.  Then we define
\[M \models p(a_0, \ldots, a_{n-1}) \iff p = (Sg)(M(A)).\]
In this case we also say that $(a_0, \ldots, a_{n-1})$ \emph{realizes} $p$ in $M$.  If there is no tuple $(a_0, \ldots, a_{n-1})$ realizing $p$ in $M$ then we say $M$ \emph{omits} $p$.

If $A \subseteq S_n$, we write $M \models A(a_0, \ldots, a_{n-1})$ to mean $M \models p(a_0, \ldots, a_{n-1})$ for some $p \in A$.
\end{defn}

%We also have the natural notion of elementary embedding.
%\begin{defn}
%Let $S$ be a type-space functor, and let $M$ and $N$ be models for $S$, with $M$ of size $\kappa$ and $N$ of size $\lambda$.  We say that an injective function $j : \kappa \to \lambda$ is an \emph{elementary embedding} if for every tuple $(a_0, \ldots, a_{n-1}) \in \kappa^n$, and every $p \in S_n$,
%\[M \models p(a_0, \ldots, a_{n-1}) \iff N \models p(j(a_0), \ldots, j(a_{n-1})).\]
%\end{defn}

\subsection{Examples}\label{sec:Examples}
Throughout the remainder of this paper we will make frequent reference to type-space functors of two particular kinds.  The first is the motivating example we have mentioned several times above, while the second provides examples built on a given topological space.

\subsubsection{First-order type-space functors}\label{sec:TypeSpace:FirstOrder}
First, we have type-space functors associated to theories in first-order logic.

\begin{defn}
Let $T$ be a first-order theory.  The \defined{type-space functor of $T$}, $S(T)$, consists of the following data.  For each $n < \omega$, let $S_n$ be the set of all complete $n$-types of $T$, considered with the logic topology (that is, the topology generated by basic closed sets of the form $[\phi] = \{p \in S_n : \phi \in p\}$ for each $n$-ary formula $\phi$).  To each $f : n \to m$ associate the map $Sf : S_m \to S_n$ defined by $(Sf)(p) = \{\phi(x_0, \ldots, x_{n-1} : \phi(x_{f(0)}, \ldots, x_{f(n-1)} \in p\}$.

We say that a type-space functor is a \defined{first-order type space functor} if it is the type-space functor of some first-order theory $T$.
\end{defn}

\begin{prop}
For every first-order theory $T$, the type-space functor of $T$ is a type-space functor.  The space $S_\omega$ is homeomorphic to the logic topology on the set of $\omega$-types of $T$.
\end{prop}
\begin{proof}
It is straightforward to verify that for any $f : n \to m$ the map $Sf$ is continuous, and that properties (1) and (2) in the definition of type-space functors hold.  To see that each $Sf$ is open, consider any basic open set $U = [\psi(x_0, \ldots, x_{n-1})] \subseteq S_n$, and any $f : n \to m$.  For each $i \in \op{im}(f)$, let $r_i \in n$ be such that $i = f(r_i)$.  If $i \not\in \op{im}(f)$, let $r_i = i$.  Let $j_1, \ldots, j_k$ enumerate $m \setminus \op{im}(f)$.  For any $q \in S_n$, we have:
\begin{align*}
q \in (Sf)[U] &\iff (\exists p \in U) q = Sf(p) \\
&\iff (\exists p \in S_m)(\psi(x_0, \ldots, x_{n-1}) \in p \text{ and } q = \{\phi(x_0, \ldots, x_{n-1}) : \phi(x_{f(0)}, \ldots, x_{f(n-1)}) \in p\}) \\
&\iff ((\exists x_{j_1} \cdots \exists x_{j_k})\psi(x_{r_0}, x_{r_1}, \ldots, x_{r_{n-1}})) \in q \\
&\iff q \in [(\exists x_{j_1} \cdots \exists x_{j_k})\psi(x_{r_0}, \ldots, x_{r_{n-1}})]
\end{align*}
Therefore $(Sf)[U]$ is a (basic) open set in $S_n$.

Now we check statement (3).  Fix $m < \omega, p \in S_m, q \in (Si_m)^{-1}(\{p\})$, and a non-empty open $U \subseteq (Si_m)^{-1}(\{p\})$.  By shrinking $U$ if necessary, we may assume that $U = [\psi(x_0, \ldots, x_m)]$ for some formula $\psi$.  Then $U$ being non-empty and included in $(Si_m)^{-1}(p)$ implies that $p~\cup~\{\psi(x_0), \ldots, x_m\}$ is consistent, so $((\exists y) \psi(x_0, \ldots, x_{m-1}, y)) \in p$.  Let $(a_0, \ldots, a_m)$ be a realization of $q$ in some model $M \models T$.  The statement $q~\in~(Si_m)^{-1}(p)$ implies that $(a_0, \ldots, a_{m-1})$ realizes $p$ in $M$.  Therefore there is $a_{m+1} \in M$ such that $M \models \psi(a_0, \ldots, a_{m-1}, a_{m+1})$.  Let $r = \tp_M(a_0, \ldots, a_{m-1}, a_m, a_{m+1})$.  Then
\begin{align*}
(Si_{m+1})(r) &= \{\phi(x_0, \ldots, x_m) : \phi(x_0, \ldots, x_m) \in r\} \\
 &= \{\phi(x_0, \ldots, x_m) : M \models \phi(a_0, \ldots, a_m)\} \\
 &= \tp_M(a_0, \ldots, a_m) \\
 &= q
\end{align*}
Also, $M \models \psi(a_0, \ldots, a_{m-1}, a_{m+1})$, so $\psi(x_0, \ldots, x_{m-1}, x_{m+1}) \in r$.  Now \\ $\psi(x_0, \ldots, x_{m-1}, x_{m+1}) = \psi(x_{d_m(0)}, \ldots, x_{d_m(m-1)}, x_{d_m(m)})$, so by definition \\ $\psi(x_0, \ldots, x_m) \in (Sd_m)(r)$, i.e., $(Sd_m)(r) \in U$.

To see that $S_\omega$ is the space of $\omega$-types of $T$, first observe that the definition of $S_\omega$ in this context gives
\[S_\omega = \{(a_0, a_1, \ldots) \in \prod_{n < \omega}S_n : \text{ for all } n < m, a_n = \{\phi(x_0, \ldots, x_{n-1}) : \phi(x_0, \ldots, x_{n-1}) \in a_m\}\}.\]
Given an $\omega$-type $p(x_0, x_1, \ldots)$, for each $n < \omega$ let $p_n$ be the restriction of $p$ to the first $n$ variables, that is,
\[p_n = \{\phi(x_0, \ldots, x_{n-1}) : \phi(x_0, \ldots, x_{n-1}) \in p\}.\]  Then the map associating $p$ to the sequence $(p_0, p_1, \ldots)$ is the required homeomorphism from the logic topology on the $\omega$-types to $S_\omega$.
\end{proof}

The model theory also agrees with classical model theory in this case.  Suppose that $T$ is a first-order theory, and $S$ is the corresponding type-space functor.  Suppose also that $\mc{M} \models T$ is enumerated as $\{m_\alpha : \alpha < \kappa\}$.  Define $M$ on $[\kappa]^{<\omega}$ by $M(\{i_0 < \ldots < i_{n-1}\}) = \tp^M(m_0, \ldots, m_{n-1})$.  It is then routine to verify that $M$ is a model (in the sense of Definition \ref{def:Model}), and that for any $p \in S_n$ and any $i_0, \ldots, i_{n-1} \in \kappa$,
\[M \models p(i_0, \ldots, i_{n-1}) \iff \mc{M} \models p(m_{i_0}, \ldots, m_{i_{n-1}}).\]
Conversely, Knight \cite[Proposition 2.10]{Knight2007} shows that every model of $S$ arises in this way from a (classical) model of $T$.

\subsubsection{Type-space functors generated by a space}\label{sec:TypeSpace:Space}
Our second class of examples of type-space functors gives examples that do not come from theories in classical logics.  We will primarily use examples of the following kind in Section \ref{sec:Distinguishing} as a source of counterexamples.

\begin{defn}
Let $X$ be any topological space.  The \defined{type-space functor of $X$}, denoted $S_X$, consists of the following data.  For each $n < \omega$, define $S_n = X^n$, and to each $f : n\to m$, associate the map $(Sf) : X^m \to X^n$ defined by $(Sf)(x_0, \ldots, x_{m-1}) = (x_{f(0)}, \ldots, x_{f(n-1)})$.
\end{defn}

\begin{prop}\label{prop:Xomega}
For any topological space $X$, the type-space functor of $X$ is a type-space functor.  The space $S_\omega$ is homeomorphic to the (Tychonoff) product topology on $X^\omega$.
\end{prop}

\begin{proof}
It is easy to see for each $f : n \to m$ the map $Sf$ is continuous and open, and that conditions (1) and (2) of the definition of type-space functors are satisfied.  To prove (3), we fix $m < \omega, p \in S_m = X^m$, $q \in (Si_m)^{-1}(\{p\})$, and a non-empty open $U \subseteq (Si_m)^{-1}(\{p\})$, and we verify $WAP_S(m, p, q, U)$.

Write $p = (p_0, \ldots, p_{m-1})$, $q = (q_0, \ldots, q_m)$.  The condition $q \in (Si_m)^{-1}(\{p\})$ means that
\[p = (Si_m)(q) = (q_{i_m(0)}, \ldots, q_{i_m(m-1)}) = (q_0, \ldots, q_{m-1}).\]
That is, $q = (p_0, \ldots, p_{m-1}, q_m)$.  We must find $r = (r_0, \ldots, r_m, r_{m+1}) \in S_{m+2} = X^{m+2}$ such that $(Si_{m+1})(r) = q$ and $(Si_{m+1})(r) \in U$.  

For each $j < m$, let $r_j = p_j$, and $r_m = q_m$.  Then the first requirement is satisfied, for we have:
\[(Si_{m+1}(r)) = (r_0, \ldots, r_m) = (p_0, \ldots, p_{m-1}, q_m) = q.\]
Now let $\pi : X^{m+1} \to X$ be the projection onto the last coordinate, and pick any $r_{m+1} \in \pi[U]$.  Since $U \subseteq (Si_m)^{-1}(\{p\})$, every element of $U$ is of the form $(p_0, \ldots, p_{m-1}, a)$ for some $a \in X$, and hence in fact $U = \{(p_0, \ldots, p_{m-1})\} \times \pi[U]$.  In particular,
\begin{align*}
(Sd_m)(r) &= (r_{d_m(0)}, \ldots, r_{d_m(m-1)}, r_{d_m(m)}) \\
 &= (r_0, \ldots, r_{m-1}, r_{m+1}) \\
 &= (p_0, \ldots, p_{m-1}, r_{m+1}) \\
 & \in U.
\end{align*}
Therefore condition (3) is satisfied as well.

For the claim about $S_\omega$, applying the definition of $S_\omega$ in this context we get
\[S_\omega = \{(a_0, a_1, \ldots) \in \prod_{n < \omega}X^n : \text{ for all } n < m, a_n = \pi_{n, m}(a_m)\},\]
where $\pi_{n, m} : X^m \to X^n$ is the projection onto the first $n$ coordinates.  The map $\theta : X^\omega \to S_\omega$ defined by $\theta(x_0, x_1, \ldots) = (x_0, (x_0, x_1), (x_0, x_1, x_2), \ldots)$ is the required homeomorphism.
\end{proof}

\section{Omitting Types}\label{sec:OTT}
Proofs of Omitting Types Theorems using the Baire Category Theorem have been given for a variety of logics; for some examples, see \cite{Morley1974}, \cite{Poizat2000}, \cite{Caicedo2014}, \cite{Eagle2014}.  In this section we describe the relationship between Baire category properties and omitting types for type-space functors.  Throughout this section, $S$ denotes a type-space functor. 

Starting from the type-space functor $S$ we will be focusing on a certain subspace $S_\mc{W}$ of $S_\omega$.  The motivation for the following definition is that we are defining an analogue of the space of $\omega$-types of the form $\tp(a_0, a_1, \ldots)$, where $(a_0, a_1, \ldots)$ enumerates a countable model of a theory.  In fact, we will see in Lemma \ref{lem:ModelsW} that there is a correspondence between elements of the following space, and the models defined in Section \ref{sec:TypeSpace:ModelTheory} above.

\begin{defn}\label{def:MSigma}
Suppose that $\sigma \in S_\omega$.  For $A \in [\omega]^n$, let $f_A : n \to \omega$ be the map sending $i$ to the $i$th element of $A$ (in increasing order).  Then we define
\[M_\sigma(A) = (Sf_A)(\sigma).\]

In the opposite direction, given a countable model $M$, for each $n < \omega$ define $\sigma_n = M(\{0, 1, \ldots, n-1\})$, and let $\sigma_M$ be the equivalence class of $(\sigma_0, \sigma_1, \ldots)$ in $S_\omega$.
\end{defn}

\begin{lem}\label{lem:Simplicial}
For each $\sigma \in S_\omega$, the map $M_\sigma$ satisfies conditions (1) and (2) of Definition \ref{def:Model}.
\end{lem}
\begin{proof}
Condition (1) is clear from the definition.  For (2), suppose that $A = \{a_0 < \ldots < a_{n-1}\} \in [\omega]^n$, $B = \{b_0 < \ldots < b_{m-1}\} \in [\omega]^m$, $A \subseteq B$, and $g : n \to m$ satisfies $a_i = b_{g(i)}$ for all $i$.  Then for each $i$,
\[f_B \circ g(i) = f_B(g(i)) = b_{g(i)} = a_i = f_A(i).\]
Therefore
\[(Sg)(M_\sigma(B)) = (Sg)((Sf_B)(\sigma)) = (S(f_B \circ g))(\sigma) = (Sf_A)(\sigma) = M_\sigma(A).\]
\end{proof}

In general we cannot expect $M_\sigma$ to be a model (that is, to satisfy condition (3) of Definition \ref{def:Model}), just as we cannot expect an arbitrary $\omega$-type of a first-order theory to specify a witness to every existential formula it implies.  We define $S_{\mc{W}}$ to be the set of those $\sigma \in S_\omega$ for which $M_\sigma$ is a model.  Formally:

\begin{defn}
For $\sigma \in S_\omega$, we put $\sigma \in S_{\mc{W}}$ if and only if for every $n < \omega$, every $A \in [\omega]^n$, every $m \geq n$, and every open $U \subseteq S_{\iota_{n, m}}^{-1}(\{M_\sigma(A)\})$, there is $B \in [\omega]^m$ and a permutation $g$ of $m$ such that $B \supseteq A$, and $(Sg)(M_\sigma)(B)) \in U$.
\end{defn}

Note that in this definition the set $U$ could equivalently be required to come from a fixed base for the topology of $S_\omega$.

\begin{lem}\label{lem:ModelsW}
The map $\sigma \mapsto M_\sigma$ is a one-to-one correspondence between $S_{\mc{W}}$ and the set of countable models of $S$, with inverse $M \mapsto \sigma_M$.
\end{lem}
\begin{proof}
Given $\sigma \in S_{\mc{W}}$, it is clear that $M_\sigma$ satisfies condition (1) of Definition \ref{def:Model}.  For condition (2), suppose that $A = \{a_0 < \ldots < a_{n-1}\} \in [\omega]^n$, $B = \{b_0 < \ldots < b_{m-1}\} \in [\omega]^m$, $A \subseteq B$, and $g : n \to m$ satisfies $a_i = b_{g(i)}$ for all $i$.  Then for each $i$,
\[f_B \circ g(i) = f_B(g(i)) = b_{g(i)} = a_i = f_A(i).\]
Therefore
\[(Sg)(M_\sigma(B)) = (Sg)((Sf_B)(\sigma)) = (S(f_B \circ g))(\sigma) = (Sf_A)(\sigma) = M_\sigma(A).\]
The definition of $S_{\mc{W}}$ exactly ensures that condition (3) is satisfied, so $M_\sigma$ is a model.  It is straightforward to check that for any $\sigma \in S_{\mc{W}}$ we have $\sigma = \sigma_{M_\sigma}$, and for any model $M$ we have $M = M_{\sigma_M}$.
\end{proof}

In light of Lemma \ref{lem:ModelsW}, we will sometimes identify a model $M$ with the sequence $\sigma_M$.

We define several omitting types properties that $S$ may have.  Another omitting types property, involving topological games, will appear in Section \ref{sec:Games}.

\begin{defn}
\begin{enumerate}
\item{
$S$ has the \emph{classical omitting types property} if for every closed non-empty $T \subseteq S_0$, and every sequence $(E_j)_{j < \omega}$ such that $E_j$ is meagre in $(S\iota_{0, j})^{-1}(T)$, there exists a model $M \models T$ such that $M$ omits every $E_j$.
}
\item{
$S$ has the \emph{strong omitting types property} if for every closed $C \subseteq S_\omega$, and every meagre $E \subseteq C$, there is a model in $C$ omitting $E$.
}
\end{enumerate}
\end{defn}

\begin{prop}\label{prop:StrongImpliesClassical}
The strong omitting types property implies the classical omitting types property.
\end{prop}
\begin{proof}
Fix a closed $T \subseteq S_0$.  To simplify notation, for $\alpha \leq \omega$, let $A_\alpha = (S\iota_{0, \alpha})^{-1}(T)$.  For each $j < \omega$ let $E_j \subseteq A_j$ be meagre.  For each $j<\omega$, and each $\mb{i} \in \omega^j$, let $f_{j, \mb{i}} : j \to \omega$ be defined by $f_{j, \mb{i}}(k) = \mb{i}_k$, where $\mb{i}_k$ is the $k$th element of $\mb{i}$ in increasing order.  Next, for each $j$ and $\mb{i}$, define
\[C_{j, \mb{i}} = (Sf_{j, \mb{i}})^{-1}(E_j).\]
Then each $\mc{C}_{j, \mb{i}}$ is meagre in $S_\mc{W} \cap A_\omega$ because $Sf_{j, \mb{i}}$ is continuous, open, and surjective. 
Finally, define
\[F = \bigcup_{j < \omega}\bigcup_{\mb{i} \in \omega^j}C_{j, \mb{i}}.\]
Then $F$ is meagre in $S_\mc{W} \cap A_\omega$.   By the strong omitting types property we can find a model $M$ such that $M$ (or, more precisely, $\sigma_M$) is in $(S_\mc{W} \cap A_\omega) \setminus F$.  For such an $M$ we have $(S\iota_{0, \omega})(M) \in T$, so $M \models T$.

To see that $M$ omits each $E_j$, suppose that $A \in \omega^j$.  Write $A = \{a_0, \ldots, a_{j-1}\} = \{c_0 < \ldots < c_{k-1}\}$, and let $g : j \to k$ be such that $c_{g(i)} = a_i$ for each $i < j$.  According to Definition \ref{def:Satisfaction}, to show that $M$ omits $E_j$ we must show that in this situation $(Sg)(M(A)) \not\in E_j$.  Unwinding Definition \ref{def:MSigma}, we obtain
\[(Sg)(M(A)) = (Sg)(M_{\sigma_M}(A)) = (Sg)(Sf_A)(\sigma_M) = S(f_A \circ g)(\sigma_M).\]
In the above calculation $f_A : k \to \omega$ sends $i$ to $c_i$, so we have $f_A \circ g(i) = c_{g(i)} = a_i$.  Letting $\mb{i} = (a_0, a_1, \ldots, a_{j-1})$ we therefore have $f_A \circ g = f_{j, \mb{i}}$.  Combining the above calculations, and using that we chose $M$ so that $(Sf_{j, \mb{i}})(\sigma_M) \not\in E_j$, we get
\[(Sg)(M(A)) = (Sf_{j, \mb{i}})(\sigma_m) \not\in E_j.\]
\end{proof}

Our omitting types properties conclude that certain countable models exist, but there are type-space functors with no countable models at all.  In order to conclude omitting types properties from topological facts about the type-space functor $S$ we must also assume that the collection of countable models for $S$ is sufficiently rich.  For type-space functors coming from countable theories this richness is provided by the downward L\"owenheim-Skolem theorem.  In general, we make the following definition.

\begin{defn}
Let $S$ be a type-space functor.  We say that $S$ has \emph{enough countable models} if $S_{\mc{W}}$ is dense in $S_\omega$.
\end{defn}

\begin{lem}\label{lem:EnoughCountableModels}
Let $S$ be a type-space functor.  If $S$ is the functor associated to a countable first-order theory, or if $S$ is generated by a separable topological space, then $S$ has enough countable models.
\end{lem}
\begin{proof}
Suppose first that $T$ is a countable first-order theory.  Then a basic open set $O \subseteq S_\omega$ is the set of all $\omega$-types of $T$ containing some particular formula $\phi$.  If $O \neq \emptyset$ then there is a model $\mc{M} \models T$ containing a tuple $\vec{a}$ satisfying $\phi$, and by L\"owenheim-Skolem we may assume $\mc{M}$ is countable.  If $\sigma$ is the type of an enumeration of $\mc{M}$ in the appropriate order (so that the elements of $\vec{a}$ have the same indices as the variables appearing in $\phi$), then $\sigma \in O \cap S_{\mc{W}}$.

Now suppose that $X$ is a space and $S = S_X$.  Let $O \subseteq X^{\omega}$ be a basic open set.  Let $D \subseteq X$ be a countable dense set, and let $\sigma \in O$ be such that every element of $D$ is listed in $\sigma$ infinitely many times.  %This is possible because $O$ imposes only finitely many constraints on $\sigma$.  
We show that $\sigma \in \mc{W}$.  We are given $n < \omega$, $m \geq n$, $A = \{a_0 < \ldots < a_{n-1}\} \in [\omega]^n$, and a non-empty basic open set $U \subseteq (S_{\iota_{n, m}}^{-1})(M_{\sigma}(A))$.  Unravelling the definitions, this means that there are open sets $V_n, V_{n+1}, \ldots, V_{m-1} \subseteq X$ such that elements in $U$ are exactly those sequences of the form $(x_0, \ldots, x_{m-1})$ where $x_i = \sigma(a_i)$ for $i < n$ and $x_i \in V_i$ for $n \leq i < m$.  Choose $B = \{b_0 < \ldots < b_{m-1}\}$ such that $b_i = a_i$ for $i < n$, and such that $\sigma(b_i) \in V_i$ for $n \leq i < m$ (this is possible by our choice of $\sigma$).  Then $M_\sigma(B) = (\sigma(b_0), \ldots, \sigma(b_{m-1})) \in U$ (so also $(Sg)(M_\sigma(B)) \in U$ where $g : m\to m$ is the identity function).  Therefore $\sigma \in S_\mc{W}$, and hence $S_\mc{W} \cap O \neq \emptyset$.
\end{proof}

The topological content of the omitting types theorem for first-order logic is captured by the following:

\begin{thm}\label{thm:OTTBaire}
Let $S$ be a type-space functor with enough countable models.  If every closed subspace of $S_{\mc{W}}$ is non-meagre in itself then $S$ has the classical omitting types property.
\end{thm}
\begin{proof}
The proof is nearly identical to the proof of Proposition \ref{prop:StrongImpliesClassical}.  As in that proof, we fix $T \subseteq S_0$ closed, and for $\alpha \leq \omega$ let $A_\alpha = (S_{\iota_{0, \alpha}})^{-1}(T)$.  For each $j < \omega$, let $E_j \subseteq A_j$ be meagre.  For each $j < \omega$ and $\mb{i} \in \omega^j$, define $f_{j, \mb{i}} : j \to \omega$ by $f_{j, \mb{i}}(k) = \mb{i}_k$, and define $C_{j, \mb{i}} = (S f_{j, \mb{i}})^{-1}(E_j)$; then each $C_{j, \mb{i}}$ is meagre in $S_{\mc{W}} \cap A_\omega$ (here we use that $S$ has enough countable models, which was not necessary in Proposition \ref{prop:StrongImpliesClassical}).  Define
\[F = \bigcup_{j < \omega}\bigcup_{\mb{i} \in \omega^j}C_{j, \mb{i}}.\]
Then $F$ is meagre in $S_{\mc{W}} \cap A_\omega$, and since $S_{\mc{W}} \cap A_{\omega}$ is non-meagre in itself by hypothesis, we can find $M \in (S_{\mc{W}} \cap A_{\omega}) \setminus F$.  This $M$ satisfies $T$ and omits each $E_j$.
\end{proof}

To characterize the strong omitting types property topologically we will need some terminology.  A topological space is \emph{completely Baire} if every closed subspace is Baire, and is \emph{completely non-meagre} if every closed subspace is non-meagre in itself.  

Hurewicz \cite{Hurewicz1928} proved that a metrizable space is completely Baire if and only if the space does not include a closed copy of $\mathbb{Q}$.  Since $\mathbb{Q}$ is meagre in itself, it follows immediately that a metrizable space is completely Baire if and only if it is completely non-meagre.  For this latter claim much weaker assumptions than metrizability are sufficient.  The one we will use is the following.

\begin{defn}
A topological space is \emph{quasi-regular} if each open set includes the closure of an open set.  A space is \emph{completely quasi-regular} if each closed subspace is quasi-regular.
\end{defn}

Quasi-regularity is commonly required to prove results about Baire spaces (see e.g. \cite{Oxtoby1957}).

\begin{lem}\label{lem:BaireNonmeagre}
A completely quasi-regular space is completely Baire if and only if it is completely nonmeagre.
\end{lem}

\begin{proof}
That completely Baire implies completely nonmeagre is immediate.  For the other direction, let $F$ be a closed subspace of a completely nonmeagre, completely quasi-regular space $X$.  Let $\{U_n\}_{n < \omega}$ be a collection of dense open subspaces of $F$.  If $\bigcap_{n < \omega} U_n$ were not dense in $F$, then there would be a $V \subseteq F$, $V$ open in $F$, such that $V \cap \bigcap_{n < \omega}U_n = \emptyset$.  Let $W$ be open in $V$ with $\overline{W} \subseteq V$.  Then $\overline{W} \cap \bigcap_{n<\omega}U_n \neq \emptyset$, because $\overline{W}$ is nonmeagre.  This contradicts $V \cap \bigcap_{n < \omega}U_n = \emptyset$.
\end{proof}

We note that regularity of type spaces can serve as a kind of weak negation.  For example, in continuous first-order logic for metric structures one does not have a classical negation, but the connective $1-x$ acts as an approximate negation, and closure under that connective is also the essential ingredient in the proof that the type spaces in continuous logic are regular.  See \cite{Caicedo1995} for more about the role of topological separation axioms in abstract model theory.  In our context we are assuming even less than regularity, though it is not clear exactly how to translate quasi-regularity into logical terms, owing to the difficulty of computing closures in the type spaces of traditional logics.

\begin{thm}\label{thm:StrongOTT}
Let $S$ be a type-space functor with enough countable models, and such that $S_{\mc{W}}$ is quasi-regular.  Then the following are equivalent:
\begin{enumerate}
\item{
$S$ has the strong omitting types property.
}
\item{
$S_{\mc{W}}$ is completely non-meagre.
}
\item{
$S_{\mc{W}}$ is completely Baire.
}
\end{enumerate}
\end{thm}
\begin{proof}
$(1) \to (2)$: Suppose that $C' \subseteq S_{\mc{W}}$ is meagre in itself, and let $C$ be a closed subset of $S_\omega$ such that $C' = S_{\mc{W}} \cap C$.  Let $E_n$ be nowhere dense in $C'$, such that $C' = \bigcup_{n < \omega}E_n$.  Then each $E_n$ remains nowhere dense in $S_\omega$, so $C'$ is meagre in $S_\omega$.  The closed set $C$ and the meagre set $C'$ contradict the statement of the strong omitting types property, because any model in $C$ is in $C \cap S_{\mc{W}}$, and therefore does not omit $C'$.

$(2) \to (3)$: Apply Lemma \ref{lem:BaireNonmeagre}.

$(3) \to (1)$: Suppose that every closed subspace of $S$ is Baire, let $C \subseteq S_\omega$ be closed, and let $E \subseteq C$ be meagre.  Let $C' = C \cap S_{\mc{W}}$, and let $E' = E \cap S_{\mc{W}}$.  By the assumption that $S$ has enough countable models, $S_{\mc{W}}$ is dense in $S_\omega$, so $C'$ is closed in $S_{\mc{W}}$ and $E'$ is meagre in $C'$.  Since $S_{\mc{W}}$ is completely Baire, $C'$ is Baire, and hence $C' \setminus E' \neq \emptyset$.  Any element of $C' \setminus E'$ corresponds to a model of $C$ omitting $E'$ (just as in the proof of Proposition \ref{prop:StrongImpliesClassical}).
\end{proof}

It is usually easier to understand the topology of $S_\omega$ than the topology of $S_{\mc{W}}$.  In concrete situations it is therefore useful to have information about how $S_{\mc{W}}$ sits as a subspace of $S_\omega$.   Recall that the \emph{weight} of a topological space $X$ is the minimum cardinality of a base for the topology of $X$.  The following lemma is very useful, and is also immediate from the definition of $S_{\mc{W}}$.
 
\begin{lem}\label{lem:SizeOfW}
Let $S$ be a type-space functor, and for each $n < \omega$ let $w(S_n)$ be the weight of $S_n$.  Then $S_{\mc{W}}$ is the intersection of $\sum_n w(S_n)$-many open subsets of $S_\omega$.
\end{lem}

\begin{ex}
Many \emph{Omitting Types Theorems} in the literature can be easily derived from Theorem \ref{thm:OTTBaire}, after translating our topological statement into model-theoretic terminology.  We present here just a few examples.  The spaces in the following examples are \emph{\v{C}ech-complete}; a space $X$ is \emph{\v{C}ech-complete} if it is a $G_\delta$ in some (equivalently, every) compactification of $X$.  Completely metrizable spaces are \v{C}ech-complete, as are locally compact Hausdorff spaces, and every \v{C}ech-complete space is Baire.
\begin{enumerate}
\item{
Let $T$ be a first-order theory in a countable language, and let $S$ be the associated type-space functor (as described in Section \ref{sec:TypeSpace:FirstOrder}).  In this context each $S_n$ is a second countable space, so $S_{\mc{W}}$ is a dense $G_\delta$ in $S_\omega$ by Lemmas \ref{lem:SizeOfW} and \ref{lem:EnoughCountableModels}.  By the compactness theorem, $S_\omega$ is a compact space, and is therefore \v{C}ech-complete.  \v{C}ech-completeness is inherited by dense $G_\delta$ subspaces, and by closed subspaces, so it follows that every closed subspace of $S_{\mc{W}}$ is \v{C}ech-complete, and hence Baire.  The Omitting Types Theorem for first-order logic then follows from Theorem \ref{thm:OTTBaire}, together with the observation that a type $p \in S_n(T)$ is principal if and only if $p$ is an isolated point of $S_n(T)$ (see \cite[Section 4.2]{Marker2002}).
}
\item{
The above discussion also works more generally, if $T$ is a theory in a countable fragment of $L_{\omega_1, \omega}$.  In this case $S_\omega$ is not compact, but it is Polish, and so is still \v{C}ech-complete.  We obtain the Omitting Types Theorem for countable fragments of $L_{\omega_1, \omega}$, originally due to Keisler \cite{Keisler1971}.  This proof of omitting types for countable fragments of $L_{\omega_1, \omega}$ is fundamentally the same as the one given by Morley \cite{Morley1974}.
}
\item{
Similarly, if $T$ is a theory in a countable fragment of the logic $L_{\omega_1, \omega}$ for metric structures, then $S_\omega$ is \v{C}ech-complete.  Translating Theorem \ref{thm:OTTBaire} into model-theoretic terminology gives the Omitting Types Theorem for (not necessarily complete) metric structures from \cite{Eagle2014}.
}
\end{enumerate}
\end{ex}

Not every omitting types theorem from the literature is a direct consequence of the topological version presented here.  Notably, the omitting types theorem for continuous logic \cite{BenYaacov2008a}, which requires that the models omitting the given types be based on \emph{complete} metric spaces, does not directly follow from our results; see \cite{Farah2014b} for a discussion of the subtleties that arise in omitting types in complete metric structures.  We note also that our topological approach to obtaining omitting types theorems bears some resemblance to Keisler's \cite{Keisler1973}, which develops both omitting types and set-theoretic forcing as a result of a more general notion of forcing that is closely related to Baire category.

\subsection{A game version of omitting types}\label{sec:Games}
The Banach-Mazur game on a topological space $X$ is a game played between two players, called EMPTY and NONEMPTY, as follows.  The players alternate choosing open sets $O_0 \supseteq O_1 \supseteq \cdots$, with EMPTY choosing first.  The player NONEMPTY wins if $\bigcap_{n < \omega}O_n \neq \emptyset$, otherwise EMPTY wins.  The connection between the Banach-Mazur game and Baire spaces is the following well-known result.

\begin{thm}[see e.g. \cite{Oxtoby1957}]
A topological space $X$ is a Baire space if and only if EMPTY does not have a winning strategy in the Banach-Mazur game.
\end{thm}

There are examples of spaces $X$ for which the Banach-Mazur game is not determined \cite{Oxtoby1957}, so asserting that NONEMPTY has a winning strategy is strictly stronger than asserting that EMPTY does not have one.  This stronger property was introduced by Choquet \cite{Choquet1969} who called it \textbf{weakly $\alpha$-favourable}.  Weak $\alpha$-favourability was further investigated by H. E. White \cite{White1975}, who, among other results, proved it was preserved by topological products –- even box products, unlike the usual Baire Category Theorem \cite{Fleissner1978}.

In light of Theorems \ref{thm:OTTBaire} and \ref{thm:StrongOTT} it is natural to ask how the Omitting Types Theorem is strengthened by using weakly $\alpha$-favourable spaces instead of Baire spaces.  By analogy to the case of first-order logic, we will refer to a closed subset of $S_\omega$ as a \emph{partial $\omega$-type}.  It is then convenient to state the Banach-Mazur game in dual form.

\begin{defn}
Let $S$ be a type-space functor, and let $C \subseteq S_{\omega}$ be a partial $\omega$-type.  The \emph{omitting types game} on $C$ is played by two players, OMIT and REALIZE, as follows.  The players alternate picking partial $\omega$-types $F_0 \subseteq F_1 \subseteq \cdots$, with REALIZE playing first, and with each $F_i$ omissible in a model realizing $C$.  The player OMIT wins if $\bigcup_{n < \omega} F_n$ is omissible in a model realizing $C$, otherwise REALIZE wins.

We say that $S$ has the \defined{game omitting types property} if OMIT has a winning strategy in the omitting types game on $C$, for every $C$.
\end{defn}

We call a space $X$ \emph{completely weakly $\alpha$-favourable} if every closed subspace of $X$ is weakly $\alpha$-favourable.  The definition of the omitting types game immediately gives the following statement, analogous to Theorem \ref{thm:StrongOTT}.

\begin{thm}\label{thm:GameOTT}
Let $S$ be a type-space functor with enough countable models.  The following are equivalent:
\begin{enumerate}
\item{
$S$ has the game omitting types property.
}
\item{
$S_{\mc{W}}$ is completely weakly $\alpha$-favourable.
}
\end{enumerate}
\end{thm}

We immediately obtain the following game version of the omitting types theorem for countable fragments of $L_{\omega_1, \omega}$, which to the best of our knowledge has not been explicitly stated elsewhere.

\begin{thm}
Let $T$ be a theory in a countable fragment of $L_{\omega_1, \omega}$.  Two players OMIT and REALIZE play the following game: REALIZE plays first, and the players alternate picking a sequence of partial $\omega$-types $\Sigma_0 \supseteq \Sigma_1 \supseteq \ldots$ (the inclusions being as sets of formulas), such that each $\Sigma_i$ is omissible in a model of $T$.  Player OMIT has a strategy to ensure that $\bigcap_{n < \omega}\Sigma_i$ is omissible in a model of $T$.
\end{thm}
\begin{proof}
In the type-space functor $S$ of $T$ the space $S_{\mc{W}}$ is Polish (see \cite{Morley1974}), and therefore completely weakly $\alpha$-favourable.  It follows that $S$ has the game omitting types property.   The statement of the game omitting types property, together with the definition of the logic topology, give the desired conclusion.
\end{proof}

In many cases of interest it is possible to deduce the game omitting types property from the topology of $S_\omega$, rather than $S_{\mc{W}}$.
\begin{thm}
Let $S$ be a type-space functor with enough countable models, such that each $S_n$ is separable and metrizable.  If $S_\omega$ is completely weakly $\alpha$-favourable then $S$ satisfies the game omitting types property.
\end{thm}
\begin{proof}
In this context $S_\omega$ is, by definition, a subspace of a product of separable metrizable spaces, and hence is itself separable and metrizable.  Moreover, $S_{\mc{W}}$ is a dense $G_\delta$ in $S_\omega$ by Lemma \ref{lem:SizeOfW} and the definition of ``enough countable models".  By Theorem \ref{thm:GameOTT} it suffices to prove the purely topological claim that if $X$ is a separable metrizable completely weakly $\alpha$-favourable space and $Y$ is a dense $G_\delta$ in $X$, then $Y$ is completely weakly $\alpha$-favourable.

Let $Z$ be a closed subspace of $Y$, and let $\overline{Z}$ be the closure of $Z$ in $X$.  Since $Y$ is metrizable and $Z$ is closed in $Y$, $Z$ is a $G_\delta$ in $Y$.  $Y$ itself is a $G_\delta$ in $X$, so $Z$ is a $G_\delta$ in $X$, and hence $Z$ is a $G_\delta$ in $\overline{Z}$.  On the other hand, $\overline{Z}$ is weakly $\alpha$-favourable by hypothesis, and of course $Z$ is dense in $\overline{Z}$.   White \cite{White1975} proved that dense $G_\delta$ subspaces of weakly $\alpha$-favourable regular spaces are weakly $\alpha$-favourable, so $Z$ is weakly $\alpha$-favourable as required.
\end{proof}
\section{Distinguishing the omitting types properties}\label{sec:Distinguishing}
In the previous section we defined three omitting types properties, so it is natural to ask if they are genuinely different.  We will focus on the question of distinguishing the game omitting types property from the strong omitting types property.  To do so we will use type-space functors of the form $S_X$, where $X$ is a  topological space (as defined in Section \ref{sec:Examples}).

\begin{lem}\label{lem:11}
Suppose that $X$ is a separable metrizable space $X$ such that $X^{\omega}$ is completely Baire, but $X^\omega$ does not include a dense completely metrizable subspace.  Then the type-space functor $S(X)$ has the strong omitting types property but does not have the game omitting types property.
\end{lem}

\begin{proof}
By Proposition \ref{prop:Xomega} $S(X)_\omega = X^\omega$, and by Lemmas \ref{lem:EnoughCountableModels} and \ref{lem:SizeOfW} $S(X)_{\mc{W}}$ is a dense $G_\delta$ in $X^\omega$.  Medini and Zdomskyy \cite{Medini2015} proved that every dense $G_\delta$ subspace of completely Baire space is completely Baire, so our assumption that $X^\omega$ is completely Baire implies that $S(X)_{\mc{W}}$ is completely Baire, and hence by Theorem \ref{thm:StrongOTT} $S(X)$ has the strong omitting types property.

Since $X$ is a separable metrizable space so is $X^\omega$, and hence also $S(X)_{\mc{W}}$.  Telg\'arsky \cite{Telgarsky1987} proved that a separable metrizable space is weakly $\alpha$-favourable if and only if it has a dense completely metrizable subspace.  Therefore if $S(X)$ had the game omitting types property, then $S(X)_{\mc{W}}$ would have a completely metrizable dense subspace, and hence $X^\omega$ would also have such a subspace, contrary to our hypothesis.  So $S(X)$ does not have the game omitting types property.
\end{proof}

We are therefore interested in the following, purely topological, problem:

\begin{prob}\label{prob1}
	Is there a separable metrizable space $X$ such that $X^\omega$ is completely Baire, but $X^\omega$ does not include a dense completely metrizable subspace?
\end{prob}

Surprisingly, the answer is positive, unless there exist large cardinals in an inner model of the set-theoretic universe. We are indebted to Lyubomyr Zdomskyy for pointing out that the problem is solved if there is a \emph{nonmeagre $P$-filter}.

\begin{defn}
	Following \cite{Bartoszynski1995} and \cite{Marciszewski1998}, given a non-principal filter $\mathcal{F}$ on an infinite countable set $T$, we regard $\mathcal{F}$ as a subspace of the copy $2^T$ of the Cantor set. $\mathcal{F}$ is a \textbf{$P$-filter} if whenever $\{U_n\}_{n<\omega}$ are members of $\mathcal{F}$, there is an $A\in\mathcal{F}$ almost included in each $U_n$. $\mathcal{F}$ is \textbf{nonmeagre} if it is a nonmeagre subspace of $2^T$.
\end{defn}

Marciszewski \cite{Marciszewski1998} proves:

\begin{lem}
	$\mathcal{F}$ is a nonmeagre $P$-filter if and only if $\mathcal{F}$ is a completely Baire space.
\end{lem}

He also proves the following ``standard fact'', which he attributes to \cite{Bartoszynski1995}:

\begin{lem}
	Let $\{\mathcal{F}_n\}_{n<\omega}$ be a sequence of $P$-filters on $\omega$. Then $\prod\limits_{n<\omega}\mathcal{F}_n$ is a $P$-filter on $\omega\times\omega$.
\end{lem}

From this he derives:

\begin{lem}
	Let $\{F_n\}_{n<\omega}$ be a sequence of nonmeagre $P$-filters on $\omega$. Then $\prod\limits_{n<\omega}\mathcal{F}_n$ is completely Baire.
\end{lem}

As Zdomskyy pointed out to us,

\begin{lem}
	If $\mathcal{F}$ is a nonmeagre $P$-filter then $\mathcal{F}^\omega$ does not include a dense completely metrizable subspace.
\end{lem}

\begin{proof}
	No filter on a countable set can include such a subspace, else it would be comeager, which is impossible. In more detail, Marciszewski points out that if $\mathcal{F}$ is a $P$-filter on $\omega$, $\mathcal{F}^\omega$ can be considered as a $P$-filter on $\omega\times\omega$. A completely metrizable subspace $M$ of $\mathcal{F}^\omega$ would be $G_\delta$, and hence if dense would be comeager. Without loss of generality, assume $M$ has no isolated points. Then $M\cap\{Z\subseteq\omega\times\omega:\langle m,n\rangle\notin Z\}$ is comeager. Then $\bigcap\{M\cap\{Z\subseteq\omega\times\omega:\langle m,n\rangle\notin Z\}:m,n\in\omega\}\neq 0$, which is absurd.
\end{proof}

By the result of Telg\'arsky \cite{Telgarsky1987} it follows that if $\mathcal{F}$ is a nonmeagre $P$-filter, then $\mathcal{F}^\omega$ is not weakly $\alpha$-favourable and neither is any dense subspace.  
We have answered Problem \ref{prob1}, assuming there is a nonmeagre $P$-filter on $\omega$.  Combining the above observations with Lemma \ref{lem:11}, we have proved:

\begin{prop}
If $\mathcal{F}$ is a nonmeagre $P$-filter on $\omega$, viewed as a topological space, then the type-space functor $S(\mc{F})$ has the strong omitting types property but does not have the game omitting types property.
\end{prop}

It remains to answer:

\begin{prob}[ \cite{Bartoszynski1995}]
	Is there a nonmeagre $P$-filter on $\omega$?
\end{prob}

This is a quite intriguing question. In \cite{Kunen2015}, seven equivalent properties of a filter on $\omega$ are given, one of which is ``nonmeagre $P$-''.

This question is discussed in \cite{Just1990}, \cite{Bartoszynski1995}, \cite{Kunen2015} and several other places. Here are some sample results taken from \cite{Bartoszynski1995}. They note that a filter $\mathcal{F}$ on $\omega$ is nonmeagre if and only if the set of functions enumerating elements of $\mathcal{F}$ is not bounded. It follows that:

\begin{prop}[ \textnormal{\cite[4.4.12]{Bartoszynski1995}}]
	If $\mathfrak{t}=\mathfrak{b}$, then there is a nonmeagre $P$-filter.
\end{prop}

Also,

\begin{prop}[ \textnormal{\cite[4.4.13 (Just et al \cite{Just1990})]{Bartoszynski1995}}]
	Suppose cof$([\mathfrak{d}]^{<\aleph_1})=\mathfrak{d}$. Then there is a nonmeagre $P$-filter.
\end{prop}

\begin{prop}[ \textnormal{\cite[4.4.14]{Bartoszynski1995}}]
	If $2^{\aleph_0}<\aleph_\omega$, then there is a nonmeagre $P$-filter.
\end{prop}

\begin{prop}[ \textnormal{\cite[4.4.15]{Bartoszynski1995}}]
	If every $P$-filter is meager, then there exists an inner model with a large cardinal.
\end{prop}

The point is that if cof$([\mathfrak{d}]^{<\aleph_1})\neq\mathfrak{d}$, then the \emph{Covering Lemma} fails. For further discussion, see \cite{Bartoszynski1995}.

Another property of a filter on the \cite{Kunen2015} list that is currently the subject of much research is \emph{countable dense homogeneity}.

\begin{defn}
	A space is \textbf{countable dense homogeneous} (briefly, \textbf{CDH}) if for every pair $\langle D,E\rangle$ of countable dense subsets of $X$, there exists a homeomorphism $h:X\to X$ such that $h[D]=E$.
\end{defn}

\cite{Hernandez-Gutierrez2014} gives a ZFC example of an uncountable CDH completely Baire subspace of $\mathbb{R}$ which is not completely metrizable. Unfortunately, as Jan van Mill pointed out to us, the space does include a dense completely metrizable subspace, and so is weakly $\alpha$-favorable. The reason is that the space (Corollary 4.6 of \cite{Hernandez-Gutierrez2014}) is the complement in the Cantor set of a particular kind of \emph{$\lambda$-set}. $\lambda$-sets have every countable subset relative $G_\delta$ and are meager. Thus their complements include dense $G_\delta$'s.

After the present work was completed, Lyubomyr Zdomskyy \cite{Talla} showed that we could replace a non-meager P-filter by the complement in the Cantor set of a non-meager Menger filter extending the Fr\'echet filter. Such a Menger filter exists in ZFC.

To definitively bury the question of whether the omitting types property is the same as Baire category, it suffices to find a space $X$ satisfying the Baire Category Theorem while the type-space functor it generates does not satisfy OTT. A Baire $X$ such that $X^\omega$ is not Baire, and hence has no dense $G_\delta$ Baire subspaces will suffice, e.g. the Baire $X$ with $X^2$ not Baire of Fleissner and Kunen \cite{Fleissner1978} will do the trick. A more nuanced example is due to Aarts and Lutzer \cite{Aarts1973}. They construct a completely Baire separable metric space with a dense completely metrizable subspace such that $X^2$ is not completely Baire. $X$ is actually weakly $\alpha$-favorable, so $X^\omega$ is as well, so $X^\omega$ is Baire, but not completely Baire.

\subsection{Definability}
Since in this note we are dealing with logics with countable signatures and hence with separable metrizable spaces, it is natural to enter descriptive set theory and ask whether there is (consistently?) a \emph{definable} (in the sense of descriptive set theory) example distinguishing the game omitting types property from the classical (or strong) omitting types property.  In this section we are considering only type-space functors generated by separable metrizable spaces.  

Recall that a subset of a Polish space is \emph{analytic} if it is the continuous image of the space of irrational numbers (which we identify with ${}^{\omega} \omega$), and is \emph{co-analytic} if its complement is analytic.  Slightly more generally, a separable metrizable space is \emph{co-analytic} if it is homeomorphic to a co-analytic subset of some Polish space.

\begin{lem}
The property of being co-analytic is preserved by the following topological operations:
\begin{enumerate}
\item{Countable intersections,}
\item{passing to open subspaces,}
\item{countable products.}
\end{enumerate}
\end{lem}
\begin{proof}
(1) is stated in \cite[p. 242]{Kechris1995}, and (2) follows from (1) plus the fact that every open set is co-analytic (again, see \cite[p. 242]{Kechris1995}).  We have been unable to find a proof in print of (3), which is undoubtedly folklore, so we sketch one for the convenience of the reader.

We first verify that a countable product of analytic sets is analytic.  An analytic set $A$ is a continuous image of ${}^\omega \omega$.  Then a countable product $\prod_{n < \omega}A_n$ of analytic sets is a continuous image of $({}^{\omega} \omega)^{\omega}$, which is homeomorphic to ${}^{\omega} \omega$.  

Now suppose that we have co-analytic sets $C_n = \mathbb{R} \setminus A_n$, where each $A_n$ is analytic.  Then $\prod_{n < \omega}C_n = \cap_{n < \omega} (\mathbb{R} \times \mathbb{R} \times C_n \times \mathbb{R} \times \cdots)$.    Since countable intersections of co-analytic sets are co-analytic, it suffices to prove that $C_n \times \mathbb{R}^\omega$ is co-analytic.  From the previous paragraph $\mathbb{R}^{\omega}$, $\mathbb{R}^{\omega+1}$, and $A_n \times \mathbb{R}^{\omega}$ are analytic.  Thus the fact that $C_n \times \mathbb{R}^{\omega}$ is co-analytic follows from the fact that $C_n \times \mathbb{R} = (\mathbb{R} \setminus A_n) \times \mathbb{R}^\omega$, which is homeomorphic to $\mathbb{R}^{\omega+1} \setminus (A_n \times \mathbb{R}^\omega)$.
\end{proof}

\begin{thm}\label{thm31}
	If $X$ is co-analytic, and $S$ is the type-space functor of $X$, then $S$ satisfies the game omitting types property if and only if it satisfies the strong omitting types property.
\end{thm}

\begin{proof}
Products of co-analytic spaces are co-analytic, so $S_\omega = X^\omega$ is co-analytic.  Open subspaces of co-analytic spaces are co-analytic, and countable intersections of co-analytic spaces are co-analytic, so since $S_{\mc{W}}$ is $G_\delta$ in $S_\omega$ (Lemma \ref{lem:SizeOfW}), $S_{\mc{W}}$ is co-analytic.

If $S$ satisfies the strong omitting types property then $S_{\mc{W}}$ is completely Baire, by \ref{thm:StrongOTT}.    By \cite[21.21]{Kechris1995}, completely Baire co-analytic sets are completely metrizable, and hence by Theorem \ref{thm:GameOTT} $S$ satisfies the game omitting types property.
\end{proof}

Theorem \ref{thm31} can be extended to all projective $X$ if one assumes the Axiom of Projective Determinacy, or a perfect set variation of the Open Graph Axiom \cite{Talla}. How definable consistent counterexamples can be is an interesting question that remains to be explored. The methods of \cite{Medini2015}, \cite{Tall}, and \cite{Talla} can be expected to be useful.
% ----------------------------------------------------------------
%\nocite{*}
\bibliographystyle{amsplain}
\bibliography{OmittingTypesBaire}

\providecommand{\bysame}{\leavevmode\hbox to3em{\hrulefill}\thinspace}
\providecommand{\MR}{\relax\ifhmode\unskip\space\fi MR }
\providecommand{\MRhref}[2]{%
  \href{http://www.ams.org/mathscinet-getitem?mr=#1}{#2}
}
\providecommand{\href}[2]{#2}
\begin{thebibliography}{10}

\bibitem{Aarts1973}
J.~M. Aarts and D.~J. Lutzer, \emph{The product of totally nonmeagre spaces},
  Proc. Amer. Math. Soc. \textbf{38} (1973), 198--200. \MR{0309056}

\bibitem{Aarts1974}
J.~M. Aarts and D.~J. Lutzer, \emph{Completeness properties designed for
  recognizing {B}aire spaces}, Dissertationes Math. (Rozprawy Mat.)
  \textbf{116} (1974).

\bibitem{Bartoszynski1995}
Tomek Bartoszy{\'n}ski and Haim Judah, \emph{Set theory}, A K Peters, Ltd.,
  Wellesley, MA, 1995, On the structure of the real line. \MR{1350295}

\bibitem{Barwise1985}
J.~Barwise and S.~Feferman (eds.), \emph{Model-theoretic logics},
  Springer-Verlag, 1985.

\bibitem{BenYaacov2003}
I.~Ben~Yaacov, \emph{Thickness, and a categoric view of type-space functors},
  Fund. Math. \textbf{179} (2003), 199--224.

\bibitem{BenYaacov2008a}
I.~Ben~Yaacov, A.~Berenstein, C.~W. Henson, and A.~Usvyatsov, \emph{Model
  theory for metric structures}, Model Theory with Applications to Algebra and
  Analysis, Vol. {II} (Z.~Chatzidakis, D.~Macpherson, A.~Pillay, and A.~Wilkie,
  eds.), Lecture Notes series of the London Mathematical Society, no. 350,
  Cambridge University Press, 2008, pp.~315--427.

\bibitem{BenYaacov2005}
Itai Ben~Yaacov, \emph{Uncountable dense categoricity in cats}, J. Symb. Logic
  \textbf{70} (2005), 829--860.

\bibitem{Caicedo1995}
X.~Caicedo, \emph{Continuous operations on spaces of structures}, Quantifiers:
  Logics, Models and Computation {I}, Synthese Library \textbf{248} (1995),
  263--263.

\bibitem{Caicedo2014}
X.~Caicedo and J.~Iovino, \emph{Omitting uncountable types, and the strength of
  $[0,1]$-valued logics}, Annals of Pure and Applied Logic \textbf{165} (2014),
  1169--1200.

\bibitem{Choquet1969}
Gustave Choquet, \emph{Lectures on analysis. {V}ol. {I}: {I}ntegration and
  topological vector spaces}, Edited by J. Marsden, T. Lance and S. Gelbart, W.
  A. Benjamin, Inc., New York-Amsterdam, 1969. \MR{0250011}

\bibitem{Eagle2014}
C.~J. Eagle, \emph{Omitting types in infinitary $[0, 1]$-valued logic}, Annals
  of Pure and Applied Logic \textbf{165} (2014), 913--932.

\bibitem{Farah2014b}
I.~Farah and M.~Magidor, \emph{Omitting types in logic of metric structures},
  2014.

\bibitem{Fleissner1978}
W.~G. Fleissner and K.~Kunen, \emph{Barely {B}aire spaces}, Fund. Math.
  \textbf{101} (1978), no.~3, 229--240. \MR{521125}

\bibitem{Hernandez-Gutierrez2014}
Rodrigo Hern{\'a}ndez-Guti{\'e}rrez, Michael Hru{\v{s}}{\'a}k, and Jan van
  Mill, \emph{Countable dense homogeneity and {$\lambda$}-sets}, Fund. Math.
  \textbf{226} (2014), no.~2, 157--172. \MR{3224119}

\bibitem{Hurewicz1928}
W.~Hurewicz, \emph{Relativ perfekte {T}eile von {P}unktmengen und {M}engen
  ({A})}, Fund. Math. \textbf{12} (1928), 78--109.

\bibitem{Just1990}
Winfried Just, A.~R.~D. Mathias, Karel Prikry, and Petr Simon, \emph{On the
  existence of large {$p$}-ideals}, J. Symbolic Logic \textbf{55} (1990),
  no.~2, 457--465. \MR{1056363}

\bibitem{Kechris1995}
A.~S. Kechris, \emph{Classical {D}escriptive {S}et {T}heory}, Springer-Verlag,
  New York, 1995. \MR{1321597 (96e:03057)}

\bibitem{Keisler1971}
H.J. Keisler, \emph{Model theory for infinitary logic: Logic with countable
  conjunctions and finite quantifiers}, North-Holland, 1971.

\bibitem{Keisler1973}
\bysame, \emph{Forcing and the omitting types theorem}, Studies in Mathematics
  (ed. M. Morley) \textbf{8} (1973), 96--133.

\bibitem{Knight2007}
R.~Knight, \emph{Categories of topological spaces and scattered theories},
  Notre Dame J. Formal Logic \textbf{48} (2007), 53--77.

\bibitem{Knight2010}
Robin Knight, \emph{Slender type categories and {V}aught's conjecture},
  preprint, 2010.

\bibitem{Kunen2015}
Kenneth Kunen, Andrea Medini, and Lyubomyr Zdomskyy, \emph{Seven
  characterizations of non-meager {$P$}-filters}, Fund. Math. \textbf{231}
  (2015), no.~2, 189--208. \MR{3361242}

\bibitem{Marciszewski1998}
Witold Marciszewski, \emph{{$P$}-filters and hereditary {B}aire function
  spaces}, Topology Appl. \textbf{89} (1998), no.~3, 241--247. \MR{1645168}

\bibitem{Marker2002}
D.~Marker, \emph{Model theory: an introduction}, Springer, 2002.

\bibitem{Medini2015}
Andrea Medini and Lyubomyr Zdomskyy, \emph{Between {P}olish and completely
  {B}aire}, Arch. Math. Logic \textbf{54} (2015), no.~1-2, 231--245.
  \MR{3304744}

\bibitem{Morley1974}
M.~Morley, \emph{Applications of topology to ${L}_{\omega_1, \omega}$},
  Proceedings, vol.~25, Amer Mathematical Society, 1974, pp.~234--240.

\bibitem{Oxtoby1957}
J.~C. Oxtoby, \emph{The {B}anach-{M}azur game and the {B}anach category
  theorem}, Contributions to the Theory of Games, Volume III (M.~Dresher, A.W.
  Tucker, and P.~Wolfe, eds.), Princeton University Press, 1957, pp.~159--163.

\bibitem{Poizat2000}
B.~Poizat, \emph{A course in model theory: an introduction to contemporary
  mathematical logic}, Springer, 2000.

\bibitem{Shelah1987}
S.~Shelah, \emph{Classification of nonelementary classes {II}: {A}bstract
  elementary classes}, Classification theory (Chicago, IL, 1985) (J.~Baldwin,
  ed.), Lecture notes in mathematics, vol. 1292, Springer, Berlin, 1987,
  pp.~419--497.

\bibitem{Tall}
F.D. Tall, S.~Todorcevic, and S.~Tokg\"oz, \emph{The {O}pen {G}raph {A}xiom and
  {M}enger's {C}onjecture}, submitted.

\bibitem{Talla}
F.D. Tall and L.~Zdomskyy, \emph{Completely {B}aire spaces, {M}enger spaces,
  and projective sets}, in preparation.

\bibitem{Telgarsky1987}
Rastislav Telg{\'a}rsky, \emph{Topological games: on the 50th anniversary of
  the {B}anach-{M}azur game}, Rocky Mountain J. Math. \textbf{17} (1987),
  no.~2, 227--276. \MR{892457}

\bibitem{White1975}
H.~E. White, Jr., \emph{Topological spaces that are {$\alpha $}-favorable for a
  player with perfect information}, Proc. Amer. Math. Soc. \textbf{50} (1975),
  477--482. \MR{0367941}

\end{thebibliography}
\end{document}